\documentclass[11pt,leqno]{amsart}
\usepackage[utf8]{inputenc}
\usepackage[utf8]{inputenc}
\usepackage{amsfonts,amssymb}
\usepackage{graphicx,tikz}
\usepackage{tikz-cd} 
\usepackage{float}
\usepackage{amsmath, amsthm, amsfonts}
\usepackage{geometry}
\usepackage{hyperref}
\usepackage{enumitem}  
\usepackage[capitalise]{cleveref}
\usepackage{comment}
\usepackage{enumitem}
\usepackage{amsrefs}
\usepackage{nicefrac}
\usepackage[official]{eurosym}
\usepackage{graphicx}
\usepackage{nomencl}
\usepackage{amsfonts}
\usepackage{hyperref}
\usepackage{amssymb}
\usepackage{fancyhdr}
\usepackage{amscd}
\usepackage[english]{babel}

\newtheorem{theorem}{Theorem}[section]
\newtheorem{thm}[theorem]{Theorem}

\newtheorem{pro}[theorem]{Proposition}

\newtheorem{lemma}[theorem]{Lemma}
\newtheorem{proposition}[theorem]{Proposition}
\newtheorem{corollary}[theorem]{Corollary}

\newtheorem*{thmA*}{Theorem A}
\newtheorem*{thmB*}{Theorem B}
\newtheorem*{thmC*}{Theorem C}
\newtheorem*{thmD*}{Theorem D}
\newtheorem*{thmE*}{Theorem E}
\newtheorem*{thmF*}{Theorem F}

\theoremstyle{definition}
\newtheorem{example}[theorem]{Example}

\theoremstyle{remark}
\newtheorem{definition}{Definition}

\newcommand{\N}{\mathbb{N}}

\newcommand{\Aut}{\mathrm{Aut}}

\newcommand{\Univ}{\mathrm{Univ}}
\newcommand{\Unim}{\mathrm{Univ^-}}
\newcommand{\Unip}{\mathrm{Univ^+}}


\title{On the directed normalizing graph associated with a group}
\author[Delizia, Gaeta, and Monetta]{Costantino Delizia, Michele Gaeta, Carmine Monetta }

\address{Dipartimento di Matematica, Universit\`a di Salerno, Fisciano 84084 (SA), Italy}
\email{(Delizia) cdelizia@unisa.it, (Gaeta) migaeta@unisa.it, (Monetta) cmonetta@unisa.it}

\begin{document}

\begin{abstract}
In this paper we investigate the \emph{directed normalizing graph} associated with a group~$G$, defined as the simple directed graph whose vertices are the elements of~$G$, with an arrow from~$x$ to~$y$ whenever the subgroup~$\langle x \rangle$ is normal in~$\langle x, y \rangle$. 
Our analysis focuses on the set of bidirectional universal vertices and, in particular, on the induced subgraph obtained by removing them, where the most interesting connectivity phenomena occur. 
We characterize the groups for which this induced subgraph is strongly connected and determine bounds for its diameter. 
Finally, we show how properties of this graph reflect algebraic features of the underlying group.
\end{abstract}

\maketitle
 
\section{Introduction}
\noindent   The study of graphs arising from groups has become a notable and expanding area within modern algebra. The interest relies on the connection between algebraic features of the group and combinatorial properties of the graph.  Readers interested in a broader overview and current open questions in this field are referred to \cites{ADM, BLN, Cameron, DNM1, DNM2, DNM3, GLM, GM, LM}.

The present work deals with a directed graph which encodes information about the lattice of normal subgroups of the associated group. The \emph{directed normalizing graph} of a group $G$ is the directed simple graph $\vec \Gamma_{\rm norm}(G)$ whose vertices are all elements of $G$, and there is a directed edge from a vertex $x$ to a vertex $y$ if the subgroup $\langle x \rangle$ is normal in the subgroup $\langle x, y \rangle$. In the sequel, we often denote the fact that $\langle x \rangle$ is normal in $\langle x, y \rangle$ by using the notation $x \rightarrow y$, and if $x \rightarrow y$ and $y \rightarrow x$ hold then we write  $x \leftrightarrow y$.

The study of the connectivity of a graph only makes sense after detecting and hence removing the set $\Univ(G)$ of all universal vertices. In general, this is far from being an easy task, and even worse when the graph is directed. Thus our first aim is to provide information about the set $\Univ(G)$, which in general is not a subgroup (see Example \ref{esempio}). We point out that, in the context of directed graphs, the term \emph{universal vertex} refers to a \emph{bidirectional universal vertex}, that is, a vertex $x$ such that $x \leftrightarrow y$ for every $y \in G$. Groups whose element are all universal in $\vec\Gamma_{\rm norm}(G)$ are precisely Dedekind groups, that is groups whose subgroups are all normal (see Theorem \ref{thm:normcompl}). Moreover, groups with only one universal vertex have been characterized (see Corollary \ref{cor:norm:A}).

Secondly, we move to the directed graph $\vec \Delta_{\rm norm}(G)$, obtained from $\vec \Gamma_{\rm norm}(G)$ by removing all bidirectional universal vertices. In particular, we study this graph when $G$ is a decomposable group, establishing conditions under which $\vec \Delta_{\rm norm}(G)$ is strongly connected and giving a sharp upper bound on its diameter.

\begin{thmA*}
    Let $G=H\times K$ be a direct product of non-Dedekind groups.  Then the graph $\vec\Delta_{\rm norm}(G)$ is strongly connected of diameter at most $3$ provided that one of the following conditions holds:
\begin{itemize}
        \item[$(i)$] $\Univ(G)=\Univ(H)\times \Univ(K)$;
         \item[$(ii)$] $(h,1)\in \Univ(G)$ if and only if $h \in Z(H)$;
        \item[$(iii)$] $(1,k)\in \Univ(G)$ if and only if $k \in Z(K)$.
\end{itemize}
\end{thmA*}

 Next, we focus on finite soluble groups $G$ having trivial center, with the aim of characterizing when  $\vec \Delta_{\rm norm}(G)$ is strongly disconnected.

\begin{thmB*}
    Let $G$ be a finite soluble group with trivial center. Then $\vec\Delta_{\rm norm}(G)$ is strongly disconnected if and only if one of the following holds:
\begin{itemize}
    \item[$(i)$] $G$ is a Frobenius group;
    \item[$(ii)$] $G$ is a $2$-Frobenius group
    with $K \lhd KH \lhd G $, $KH$ a Frobenius group and $G/K$ a Frobenius group with kernel $KH/K$
    such that  $p  \nmid r-1$ for all $p \in \pi(H)$ and for all $r \in \pi(K)$.
    \end{itemize}
\end{thmB*}

As a consequence of the main result in \cite{parker}, we give a general bound on the diameter of $\vec\Delta_{\rm norm}(G)$, proving the following.

\begin{thmC*}
Let $G$ be a finite soluble group with trivial center. 
\begin{itemize}
    \item[$(i)$] If $\vec\Delta_{\rm norm}(G)$ is strongly connected, then the diameter of $\vec\Delta_{\rm norm}(G)$ is at most $8$.
    \item[$(ii)$] If $\vec \Delta_{\rm norm} (G)$ is strongly disconnected, then the number of strongly connected components is $|\mathrm{Fit}(G)|+1$; moreover, one strongly connected component has diameter at most $6$ and all other strongly connected components have diameter at most $2$.
\end{itemize}
\end{thmC*}

It is worth mentioning that tighter bounds on the diameter of $\vec\Delta_{\rm norm}(G)$ can be obtained by imposing additional conditions on $G$.
For instance, we prove that ${\rm diam} (\vec\Delta_{\rm norm}(G)) \leq 4$ provided that either $G$ is cyclic-by-abelian (see Proposition \ref{prop:norm:cyclic-by-abelian}), or the Fitting subgroup of $G$ has prime index (see Proposition~\ref{prop:norm:primeindex}).  

Finally, we also study the undirected normalizing graph, denoted by  $\Gamma_{\rm norm}(G)$, which is the undirected graph induced by $\vec \Gamma_{\rm norm}(G)$. This graph has appeared in the literature very recently. More specifically, in \cite{FarrellParker2025} Farrell and Parker classify  when the subgraph induced by $\Gamma_{\rm norm}(G)$ on $G\setminus \{1\}$ is connected, also giving a sharp upper bound on its diameter provided the group is soluble with trivial center. Our main result in this direction shows that $G$ is nilpotent of class at most $3$ whenever $\Gamma_{\rm norm}(G)$ is complete (see Theorem \ref{teo:norm:nilp}).


\section{Preliminaries}
\noindent We begin by recalling some basic definitions and notation related to directed graphs, which will serve as the foundational framework for the results presented in the following sections.

Let $ \vec \Gamma = (V, E) $ be a directed graph.
If $(u,v) \in E$ we write $u \to v$ and we often refer to $u$ as the tail of the arc, and to $v$ as the target. If we have $u \to v$ and $v \to u$ we use the notation $u \leftrightarrow v$. We say that $u$ and $v$ are adjacent in $\vec \Gamma$ if $u \to v$ or $v \to u$.
A simple directed graph $\vec \Gamma=(V,E)$ is \emph{complete} if for every pair of distinct vertices $u,v\in V$ we have 
$u \leftrightarrow v$.

Particular kinds of vertices of directed graphs are defined by the notions of sink and source, where a vertex $v \in V$ is called a \emph{sink} if there are no edges in $E$ with $v$ as a source and it is called a \emph{source} if there are no edges in $E$ with $v$ as a target.
A vertex is called isolated if it is both a sink and a source.
The same notions can be extended to subsets of the vertex set.
Thus, we say that a subset of vertices $S \subseteq V$ is a \emph{sink set} if there are no edges from any vertex in $S$ to any vertex in $V \setminus S$, while it is a \emph{source set} if there are no edges from any vertex in $V \setminus S$ to any vertex in $S$.

Recall that a \emph{directed path} of length $k \geq 1$ in a directed graph $\vec \Gamma=(V,E)$  is a $k+1$-tuple of vertices $P= (v_1, v_2, \dots, v_{k+1})$ such that $v_i \to v_{i+1}$ for each $i = 1, \dots, k$. A path is called \emph{simple} if it contains no repeated vertices. To represent a path $P= (v_1, v_2, \dots, v_k, v_{k+1})$ we use the notation $v_1 \to v_2 \to \dots \to v_{k+1}$. Additionally, we say that the vertex $v_1$ reaches the vertex $v_{k+1}$ in $k$ steps provided the existence of a path  $P= (v_1, v_2, \dots, v_k, v_{k+1})$.

Directed paths give rise to one of the notions we study the most, which is the one of strongly connected graph. We say that
a directed graph $\vec \Gamma=(V,E)$ is \emph{strongly connected} if for every pair of vertices $u,v \in V$, there exist a directed path from $u$ to $v$ and a directed path from $v$ to $u$.
Moreover,  for any two distinct vertices $u,v\in V$ we define the \emph{directed distance}
\[
\vec d(u,v)
=
\min\bigl\{\ell \geq 1 : \text{there exists a directed path of length $\ell$ from $u$ to $v$}\bigr\}.
\]
In other words the directed distance from $u$ to $v$ is the length of the shortest path that connects $u$ to $v$.
Since $\vec \Gamma$ is strongly connected, $\vec d(u,v)$ is well defined for all ordered pairs $(u,v)$. Notice that, in general, $\vec d(u,v)\neq \vec d(v,u)$.
If a directed graph is strongly connected we can define its \emph{diameter}, denoted by $\mathrm{diam}(\vec \Gamma)$, as
\[
\mathrm{diam}(\vec \Gamma) = \sup \bigl\{ \vec d(u,v) \mid u,v \in V, u\neq v \bigr\}.
\]
When a directed graph is not strongly connected, it can be broken into smaller pieces called strongly connected components.
A \emph{strongly connected component} is a maximal subgraph that is strongly connected. That is, it is a subgraph such that any two vertices in the subgraph are connected by a path, and no additional vertices from the graph can be added without losing the property of strong connectivity.
Of course, one can talk about the diameter of any strongly connected component of a directed graph.

We will also frequently refer to Frobenius and $2$-Frobenius groups. Thus we mention some well-known facts.

\begin{lemma}\label{lem:norm1}
    Let $G=KH$ be a Frobenius group, with kernel $K$ and complement $H$. Then $N_G(\langle h \rangle ) \subseteq H$ for all $h \in H \setminus \{1\}$.
\end{lemma}

Moreover we fix the following notation.

\begin{definition}\label{2-Frob}
   A group $G$ is a \emph{$2$-Frobenius group} if
 there exist three subgroups $K,H, L $  of $G$ such that $K$ and $KH$ are normal in $ G $,  $KH$ is a Frobenius group with kernel $K$ and complement $H$ and $G/K$ is a Frobenius group with kernel $KH/K$ and complement $L/K$. Moreover, we set $X=\bigcup_{g \in G} H^g$.
\end{definition}

 For completeness we include the proofs of the following result.

\begin{lemma}\label{lem:norm3}\label{lem:norm2}
    Let $G$ be a $2$-Frobenius group as in Definition \ref{2-Frob}. Then:
    \begin{itemize}
        \item[$(i)$] $H$ is cyclic of odd order;
        \item[$(ii)$] for all $g \in G\setminus HK$ such that $o(g)=p$, with $p$ prime, $C_K(g)\neq \{1\}$.
        \item[$(iii)$] $\bigcup_{g\in G}H^g=\bigcup_{k\in K}H^k$
    \end{itemize}
\end{lemma}
\begin{proof}
   For proof of $(i)$ and $(ii)$ see \cite[Lemma 2.2 and Lemma 3.8]{CL}. For $(iii)$ we argue as follows. Let $g \in G$. Since $KH$ is normal in $G$ we have $H^g \subseteq HK$ and so $\bigcup_{g\in G}H^g\subseteq KH$. Moreover, $KH$ is a Frobenius group with kernel $K$ and thus $\bigcup_{g\in G}H^g\subseteq \bigcup_{g\in KH}H^g=\bigcup_{k\in K}H^k$. The other inclusion is trivial.
\end{proof}

 Lastly, we fix the following notation. If a group $G$ has an element $x$ whose order is divisible by a prime $p$, we denote by $x_p$ the $p$-component of $x$ and with $x_{[p]}$ the power of $x$ of order $p$. Moreover, the symmetric group of degree $n$, the dihedral group of order $2n$, the quaternion group of order $2^n$ will be denoted by $S_n$, $D_{2n}$ and $Q_{2^n}$, respectively.

\section{Universal vertices}
\noindent In this section, we study the set of bidirectional universal vertices of $\vec \Gamma_{\rm norm}(G)$, namely $\Univ(G)$, giving some useful properties satisfied by this set and characterizing groups $G$ for which $\Univ(G)=\{1\}$.

   Studying graphs it is quite natural to ask how the neighborhood of an element looks like. We set $N^+(x)=\{ y \in G \mid \langle x \rangle \unlhd \langle x, y \rangle \}$ and $N^-(x)=\{y \in G \mid \langle y \rangle \unlhd \langle x, y \rangle\}$. 
    \begin{pro}\label{prop:norm:vicinato}
        Let $G$ be a group and let $x \in G$. Then: 
        \begin{itemize}
            \item[$(i)$] $N^+(x)=N_G(\langle x\rangle)$;
            \item[$(ii)$] $\{a \in G \mid \langle a \rangle \unlhd G\} \cup C_G(x) \subseteq N^-(x)$.
        \end{itemize} 
    \end{pro}
\begin{proof}
  Statement $(i)$ follows from the fact that $\langle x \rangle$ is normal in $\langle x,y \rangle$ if and only if $y$ normalizes $\langle x \rangle$. 
   
  We now prove statement $(ii)$.  Let $a \in G$ such that $\langle a \rangle $ is normal in $G$. Then $\langle a \rangle$ is normal in $\langle a, x \rangle$ and so $a \rightarrow x$. Thus $a \in N^-(x)$. If $a\in C_G(x)$ then $a$ commutes with $x$ and thus $a \rightarrow x$.
\end{proof}

Notice that, in general,  the set $N^-(x)$ is not a subgroup of $G$ and also that the inclusion in $(ii)$ can be strict. As an example take the symmetric group $S_3$ and the transposition $x=(1 2)$. Then $N^-(x)=\{1, (1 2), (123), (132)\}$, which is not a subgroup of $S_3$.  Furthermore, consider the symmetric group $S_4$ and $x=(3 4)$. Then $(243) \in N^-(x)$ but $\langle (243) \rangle$ is not normal in $S_4$ and $(234) \not \in C_{S_4}(x)$.

Linked to both sets of neighborhoods one can consider the sets of universal forward and of universal backward vertices. Denote
$\Unim(G)=\{ x \in G \mid x \rightarrow y, \mbox{ for every } y \in G\}$ the set of universal backward vertices, and
$\Unip(G)=\{ x \in G \mid y \rightarrow x, \mbox{ for every } y \in G\}$ the set of universal forward vertices.

\begin{pro}\label{prop:norm:uni-uni+}
        Let $G$ be a group and let $x \in G$. Then:
        \begin{itemize}
            \item[$(i)$] $\Unim(G)=\{a \in G \mid \langle a \rangle \unlhd G\}$;
            \item[$(ii)$] $\Unip(G)= \bigcap_{a \in G} N_G(\langle a \rangle).$
        \end{itemize}
    \end{pro}
\begin{proof}
We first prove $(i)$, noticing that $x \in \Unim(G)$ if and only if  $N^+(x)=G$, so, by $(i)$ of Proposition \ref{prop:norm:vicinato}, if and only if $N_G(\langle x \rangle)=G$, that is $\langle x \rangle \unlhd G$.

  We now prove  $(ii)$. An element $x\in G$ belongs to $ \bigcap_{a \in G} N_G(\langle a \rangle)$ if and only if it normalizes every $a \in G$ and thus if and only if $x \in \Unip(G)$.
    \end{proof}

Notice that $\Unip(G)$ is always a subgroup of $G$. In fact, it is the Baer norm of $G$, which is the intersection of the normalizers of all  subgroups of $G$. It was introduced by Baer in \cite{BaerNorm}. Several properties are satisfied by this subgroup. For example, it has been shown that this is always a characteristic subgroup of $G$ which is contained in the second term of the upper central series of $G$. Meanwhile, the set $\Unim(G)$ is not a subgroup of $G$, as one can see in Example \ref{esempio}.

We denote by $\Univ(G)=\Unim(G)\cap \Unip(G)$ the set of all bidirectional universal vertices. When not specified, we will refer to bidirectional universal vertices as universal vertices.

The following result is a direct consequence of a result in \cite{S}. We write $Z_2(G)$ for the second center of a group $G$.

\begin{proposition}\label{prop:norm:ZUF}
    Let $G$ be a group. Then $Z(G) \subseteq \Univ(G) \subseteq Z_2(G)$. 
\end{proposition}
\begin{proof}
    First notice that every element of $Z(G)$ commutes with every element of $G$ and thus it is trivially universal. If $x \in \Univ(G)$ then $x \in \Unip(G)$ which is contained in $Z_2(G)$ by Theorem in \cite{S}. This concludes the proof.
\end{proof}

\begin{corollary}\label{cor:norm:A}
    Let $G$ be a group. Then $\Univ(G)=\{1\}$ if and only if $Z(G)=\{1\}$.
\end{corollary}

We observe that in general $ Z(G) \neq \Univ(G) \neq Z_2(G)$; moreover, $\Univ(G)$ need not be a subgroup of $G$.

\begin{example}\label{esempio}
In fact, let $G= \langle x,y \mid x^8=y^2=1, x^y=x^5\rangle$. Then $Z(G)=\{1, x^2, x^4, x^6\}$ has order $4$. Consider $x^2y \in G \setminus Z(G)$. We have $H=\langle x^2y\rangle=\{1, x^2y, x^4, x^6y\}$ which is normal in $G$.    Furthermore, $(x^ny)^{x^2y}=x^{5n}y=(x^ny)^5 \in \langle x^ny \rangle$. Therefore, $x^2y \in \Univ(G)\setminus Z(G)$. Moreover, since $y^x=x^6y \not \in \langle y \rangle$, it follows that $x, y \in G \setminus \Univ(G)$ and so $\Univ(G) \neq G=Z_2(G)$.  It is not difficult to see that  $\Univ(G)=Z(G)\cup \{x^2y, x^4y\}$ and thus it has order $6$ and it is not a subgroup of $G$.  In general, the fact that $\Univ(G)$ is not a subgroup of $G$  also shows that $\Unim(G)$ need not be a subgroup of $G$.
\end{example}

Now we go further in our investigation about universal vertices. 

\begin{lemma}\label{lem:norm:powerconj}\label{lemma:norm:orderprime}\label{cor:norm:centro}
    Let $G$ be a group and $x \in \Univ(G)$. Then:
    \begin{itemize}
        \item[$(i)$] $x^m\in \Univ(G)$ for any integer $m$; 
        \item[$(ii)$] $x^g \in \Univ(G)$ for any $g\in G$;
        \item[$(iii)$] if $x$ has finite order then for any prime $p \in \pi(o(x))$ there exists an element of $ \Univ(G)$ having order $p$; 
        \item[$(iv)$] if $x$ has order $2$, then $Z(G) \neq \{1\}$.
    \end{itemize} 
\end{lemma}
\begin{proof}
Let $m$ be an integer and $g \in G$. Then there exists a suitable integer $n$ such that $(x^m)^g =(x^g)^m=x^{nm}=(x^m)^n \in \langle x^m \rangle$. This proves $(i)$. Furthermore, conjugating $g$ by $x$ gives a power of $g$ and thus we have also $g^{x^m}\in \langle g \rangle$.

  We now prove $(ii)$. Since $\langle x \rangle$ is normal in $G$ we have $x^g=x^m$ for some integer $m$. Thus, $x^g$ is a power of a universal vertex and so by statement $(i)$ we have $x^g \in \Univ(G)$.

Statement $(iii)$ is an immediate consequence of $(i)$.

Lastly, we prove statement $(iv)$.  Let $x \in \Univ(G)$  of order $2$. In particular, for any $g \in G$ we have that $ g$ normalizes $\langle x \rangle$, and thus it centralizes $x$. It follows that $x \in Z(G)$.
  
\end{proof}

From the previous result it follows that $\Univ(G)$ is a union of conjugacy classes.

\begin{lemma}
    Let $G$ be a finite group and let $x, y \in \Univ(G)$ of coprime order such that $[x,y]=1$. Then $xy\in \Univ(G)$. 
\end{lemma}
    \begin{proof}
       Let $g \in G$. From $x,y \in \Univ(G)$ we have $x,y \in N_G(\langle g \rangle)$. It follows that $xy \in N_G(\langle g \rangle)$. Consider now $(xy)^g$. We have 
        \[
        (xy)^g=x^gy^g=x^ny^m,
        \]
        for some integers $n$ and $m$. Since $(o(x), o(y))=1$ there exists an integer $k$ such that $k \equiv n \pmod{o(x)}$ and  $k \equiv m \pmod{o(y)}$. Thus $x^ny^m=x^ky^k$. Since $x$ and $y$ commute we have $x^ky^k=(xy)^k$. Therefore $g$ normalizes $\langle xy \rangle$ and $xy \in \Univ(G)$.
    \end{proof}

It is possible to generalize statement $(iv)$ of Lemma \ref{lem:norm:powerconj} showing that every element of prime order of $\Univ(G)$ is central.

\begin{proposition}\label{prop:norm:fundamental}
    Let $G$ be a finite group and let $x \in \Univ(G)$. If $x$ has prime order, then $x \in Z(G)$.
\end{proposition}
\begin{proof}
    By Lemma \ref{lemma:norm:orderprime} we can assume that $x$ has order $p$, where $p$ is an odd prime. Let $C=C_G(\langle x\rangle)$. Suppose by contradiction that $x \not \in Z(G)$. Then $C$ is a proper subgroup of $G$. Since $\langle x \rangle \unlhd G$, we have $N_G(\langle x \rangle)/C=G/C$ and it is isomorphic to a nontrivial subgroup of $\Aut(\langle x \rangle )$ which is cyclic of order $p-1$. Thus $G/C$ is cylic.  Let $d=|G/C|=q_1^{\alpha_1}\dots q_s^{\alpha_s}$, with $q_i$ different primes  and $\alpha_i \in \N$ for all $i$. Consider $z \in G \setminus C$. Then $o(zC)$ divides $d$. Suppose $t$ is a prime that divides $o(z)$ but not $o(zC)$ and say $t^\ell$ the maximum power of $t$ that divides $o(z)$. Then $z^{t^\ell}C\neq C$, since $t$ and $o(zC)$ are coprime. Thus, there exists $y \in G\setminus C$ of order $q_1^{r_1}\dots q_s^{r_s}$, with $r_i \in \N$ for all $i$. However, $x$ normalizes $\langle y \rangle$ and so there exists an homomorphism from $\langle x \rangle$ to $\Aut(\langle y \rangle )$. Since $|\Aut(\langle y \rangle )|=q_1^{r_1-1}(q_1-1)\dots q_s^{r_s-1}(q_s-1)$ we have $(p, |\Aut(\langle y \rangle )|)=1$. Therefore, the only possibility is that $x$ centralizes $y$, which is a contradiction.  
\end{proof}

We point out that from Proposition \ref{prop:norm:fundamental} and  Lemma \ref{lem:norm:powerconj} $(i)$ it is possible to obtain an alternative proof of Corollary \ref{cor:norm:A} for finite groups, which avoids using the result in \cite{S}.

The following result helps us to describe the set $\Univ(G)$ when $G$ is a decomposable finite group.

\begin{pro}\label{prop:norm:univ}
    Let $G=H\times K$ be a finite group.  Then:
    \begin{itemize}
        \item[$(i)$] $\Unip(G)\subseteq \Unip(H)\times \Unip(K)$;
        \item [$(ii)$] $\Unim(G)\subseteq \Unim(H)\times \Unim(K)$;
        \item[$(iii)$] $ \Univ(G) \subseteq \Univ(H)\times \Univ(K)  $;
        \item[$(iv)$] if $(|H|,|K|)=1$ then $ \Univ(G) = \Univ(H)\times \Univ(K)  $;
        \item[$(v)$]if $Z(H)=\Univ(H)$ and $Z(K)=\Univ(K)$ then $ \Univ(G) = \Univ(H)\times \Univ(K)$.
    \end{itemize}
\end{pro}
\begin{proof}
    We first prove statement $(i)$. Let $g \in \Unip(G)$. Then there exist $h \in H$ and $k \in K$ such that $g=(h,k)$. Let $h_1 \in H$ and $k_1 \in K$. Then $(h_1, k_1)\in G$, thus $(h_1,k_1)^g=(h_1,k_1)^n=(h_1^n,k_1^n)$, for some integer $n$. However, $(h_1,k_1)^g=(h_1,k_1)^{(h,k)}=(h_1^h,k_1^k)$. Therefore $h_1^h=h_1^n$, $k_1^h=k_1^n$ and so $h \in \Unip(H)$ and $k \in \Unip(K)$. 

    For statement $(ii)$ we argue similarly. Let $g \in \Unim(G)$. Then there exist $h \in H$ and $k \in K$ such that $g=(h,k)$. Let $h_1 \in H$ and $k_1 \in K$. Then $(h_1, k_1)\in G$, thus 
    $(h,k)^{(h_1,k_1)}=(h,k)^n=(h^n,k^n)$ for some integer $n$. However, $(h,k)^{(h_1,k_1)}=(h^{h_1},k^{k_1})$. Therefore $h^{h_1}=h^n$ and $k^{k_1}=k^n$ and so $h \in \Unim(H)$ and $k \in \Unim(K)$.

Statement $(iii)$ follows from $(i)$ and $(ii)$. 
   
    We now prove $(iv)$. By $(iii)$, we only need to prove one inclusion. Let $h \in \Univ(H)$ and $k \in \Univ(K)$. Let $g=(h_1,k_1)\in G$. Then $(h,k)^g=(h^{h_1},k^{k_1})=(h^n,k^m)$ for some integers $n$ and $m$. However, since $o (h)$ and $o(k)$ are coprime, there exists $x \in Z$ such that $x \equiv n \pmod {o(h)}$ and $x \equiv m \pmod {o(k)}$. Therefore $(h,k)^g=(h^n, k^m)=(h^x,k^x)=(h,k)^x$ and $(h,k)\in \Unim(G)$. Moreover, $g^{(h,k)}=(h_1^s,k_1^t)$ for some integers $s$ and $t$. As before there exists an integer $y$ such that $(h_1^s,k_1^t)=(h_1,k_1)^y$. Thus $(h,k)\in \Unip(G)$ and we are done.

    Lastly, we prove $(v)$. By $(iii)$ we only need to prove one inclusion. By hypothesis $\Univ(H) \times \Univ(K)=Z(H)\times Z(K)=Z(G)\subseteq \Univ(G)$.
\end{proof}

In general the inclusions in $(i), (ii)$ and $(iii)$ of Proposition \ref{prop:norm:univ} can be strict. For instance, consider $G=Q_8 \times Q_8$. Then $\Univ(G)=Z(G)$ which has order $4$, while $Q_8$ is a Dedekind group, and thus $\Univ(Q_8)=Q_8$.

The next two results will be very useful in the study of the diameter of the graph $\vec \Delta_{\rm norm}(G)$, for a decomposable group $G$. 

\begin{lemma}\label{lem:norm:univinH}
Let $G=H\times K$, where $H$ is a non-abelian finite group. Let $x \in \Univ(H)\setminus Z(H)$ and $y \in H$, and let $n$ be an integer such that $y^x=y^n\neq y$. If there exists $k \in K$ such that $o(k)=o(y)$ then $(x,1)\not \in \Univ(G)$.
\end{lemma}
\begin{proof}
   By hypothesis, $(y,k)^{(x,1)}=(y^n,k)$. Since $o(k)=o(y)$ there are no integers $m$ such that $m \equiv n \pmod{o(y)}$ and $m \equiv 1 \pmod{o(k)}$, and thus $(y^n,k)\not \in \langle (y,k) \rangle$. Therefore $(x,1) \not \in \Univ(G)$.
\end{proof}

\begin{lemma}\label{lem:norm:univ-}
    Let $G=H\times K$ be a group. Then:
    \begin{itemize}
        \item[$(i)$] if $(a, 1) \in \Univ(G)$ and $(1,b) \in \Univ(G)$, then $(a,b) \in \Unip(G)$;
        \item[$(ii)$] if $a \in \Univ(H)$, then $(a,1) \in \Unim(G)$;
        \item[$(iii)$] if $b \in \Univ(K)$, then $(1,b) \in \Unim(G)$.
    \end{itemize}
\end{lemma}
\begin{proof}
 Statement $(i)$ follows from the fact that  $(a, 1) \in \Unip(G)$,  $(1,b) \in \Unip(G)$ and $\Unip(G)$ is a subgroup of $G$.

We now prove $(ii)$. Let $a \in \Univ(H)$ and $(c,d) \in G$. Then $(a,1)^{(c,d)}=(a^c,1)=(a^n,1)$, for some integer $n$. Thus, $(a,1)^{(c,d)}=(a,1)^n\in \langle (a,1) \rangle$ and so $(a,1) \in \Unim(G)$.

The statement $(iii)$ is proven in the same way as $(ii)$.

\end{proof}

\section{Completeness}
\noindent Let $G$ be a group. In this section, we focus on the graphs $\vec\Gamma_{\rm norm}(G)$ and the undirected graph induced by it, namely $\Gamma_{\rm norm}(G)$, with the aim of understanding when these are complete. In particular, we characterize when  $\vec\Gamma_{\rm norm}(G)$ is complete for any group $G$ and, using completeness of  $\Gamma_{\rm norm}(G)$, we provide a sufficient condition for a group to be nilpotent.

\begin{thm}\label{thm:normcompl}
   A group $G$ is Dedekind if and only if $\vec \Gamma_{\rm norm}(G)$ is complete.
\end{thm}
\begin{proof}
If $G$ is a Dedekind group, then for every $x \in G$ we have $\langle x \rangle$ is normal in $G$. Therefore for every $x,y \in G$ we have $\langle x \rangle$ is normal in $\langle x,y \rangle$. Thus $\vec \Gamma_{\rm norm}(G)$ is complete.

On the other hand, assume that $\vec \Gamma_{\rm norm}(G)$ is complete. Let $H$ be a subgroup of $G$. For every $h \in H$ and for every $g \in G$ we have $\langle h \rangle$  is normal in $\langle h, g \rangle$. Thus $h^g \in H$, so $H$ is normal in $G$.
\end{proof}

Let now consider the undirected graph $\Gamma_{\rm norm}(G)$ whose vertices are the element of the group $G$ and where $x$ and $y$ are adjacent in $\Gamma_{\rm norm}(G)$  if and only if either $x \to y$ or $y \to x$ in $\vec \Gamma_{\rm norm}(G)$. Recall that the supersolubility graph of a group $G$ is the undirected simple graph obtained taking as vertices all elements of $G$ and drawing an edge between two elements $x,y \in G$ if and only if the subgroup $\langle x, y \rangle$ is supersoluble.

\begin{pro}\label{prop:norm:supersolv}
Let $G$ be a group. Then $\Gamma_{\rm norm}(G)$ is a subgraph of the supersolubility graph of $G$.

\end{pro}
\begin{proof}
   If $x$ and $y$ are adjacent in $\Gamma_{\rm norm}(G)$ then $\langle x \rangle$ is normal in $\langle x, y \rangle$ or $\langle y \rangle$ is normal in $\langle x,y \rangle$. Without loss of generality suppose that $\langle x \rangle$ is normal in $\langle x, y \rangle$. Then $1 \leq \langle x \rangle \leq \langle x, y\rangle$ is a normal series with cyclic factors. Therefore $x$ and $y$ generate a supersoluble group and thus they are adjacent in the supersolubility graph of $G$.
\end{proof}

For a finite group $G$ it immediately follows from Proposition \ref{prop:norm:supersolv} and  \cite[Theorem 4.8]{CFH} that $G$ is supersoluble provided $\Gamma_{\rm norm}(G)$ is complete. Actually it is possible to say something more.

\begin{thm}\label{teo:norm:nilp}
    Let $G$ be a finite group with $\Gamma_{\rm norm}(G)$ complete. Then:
    \begin{itemize}
        \item[$(i)$] $G$ is nilpotent of nilpotency class at most $3$;
        \item[$(ii)$] all involutions of $G$ commute.
    \end{itemize}
\end{thm}
\begin{proof}
We first prove $(i)$. Let $x,y \in G$. Since $\Gamma_{\rm norm}(G)$ is complete we have $\langle x \rangle \unlhd \langle x, y \rangle$ or $\langle y \rangle \unlhd \langle x, y \rangle$. Without loss of generality suppose $\langle x \rangle \unlhd \langle x, y \rangle$. Moreover, we have $\langle y \rangle \unlhd \langle xy, y \rangle$ or $\langle xy \rangle \unlhd \langle xy, y \rangle=\langle x, y \rangle $. Therefore, it follows that $\langle x, y \rangle=\langle x \rangle \langle y \rangle$ or $\langle x, y \rangle=\langle x \rangle \langle xy \rangle$. In either case,   $\langle x, y \rangle$ is a product of cyclic normal subgroups and thus it is nilpotent of class at most $2$. Thus applying the result in \cite{Levi} the proof is completed.

  We now prove $(ii)$. Let $g,h$ be two involutions of $G$. Since $\Gamma_{\rm norm}(G)$ is complete it follows that $g$ normalizes $\langle h \rangle$ or $h$ normalizes $\langle g \rangle$. Since both $h$ and $g$ have order $2$, we are done. 
\end{proof}

In general a nilpotent group of class $2$ does not have a complete normalizing graph, as the dihedral group of order $8$ shows.

We also point out that the completeness of $\Gamma_{\rm norm}(G)$ does not imply that  $\vec \Gamma_{\rm norm}(G)$ is complete. As a counterexample, one can take $G$ as in Example \ref{esempio}. In fact, $\Unip(G)$ has index $2$ in $G$ and so by Corollary \ref{cor:norm:completesenzafreccia} it follows that $\Gamma_{\rm norm}(G)$ is complete.

The following provides a sufficient condition for $\Gamma_{\rm norm}(G)$ to be complete.

\begin{lemma}\label{lem:norm:completeness}
    Let $G$ be a group and $x \in G\setminus \Unip(G)$. If for all $y \in G\setminus \Unip(G)$ there exists $c \in C_G(x)$ such that $c^{-1}y \in \Unip(G)$  then $\langle x \rangle$ is normal in $G$. 
\end{lemma}
\begin{proof}
   We prove that any $y \in G$ normalizes $\langle x \rangle$. This is obviously true if $y \in \Unip(G)$. Now assume $y \in G\setminus \Unip(G)$. By hypothesis, there exists $c \in C_G(x)$ such that $c^{-1}y\in \Unip(G)$. Thus  $x^y=x^{c^{-1}y} \in \langle x \rangle$. Therefore, $\langle x \rangle$ is normal in $G$.
\end{proof}

\begin{proposition} \label{prop:norm:completeness} 
    If $G$ is a group such that for any $x,y \in G\setminus \Unip(G)$ there exists $c \in C_G(x)$ such that $c^{-1}y \in \Unip(G)$  then $\Gamma_{\rm norm}(G)$ is complete. 
\end{proposition}

\begin{proof}
    Let $x,y \in G$. If either $x$ or $y$ belongs to $\Unip(G)$ then $x$ and $y$ are adjacent in $\Gamma_{\rm norm}(G)$. Otherwise apply Lemma \ref{lem:norm:completeness} to obtain the result.
\end{proof}

\begin{corollary}\label{cor:norm:completesenzafreccia}
    If $\Unip(G)$ has prime index in a group $G$ then $\Gamma_{\rm norm}(G)$ is complete.
\end{corollary}
\begin{proof}
   Since every $x\in G\setminus \Unip(G)$ is a generator of $G$ modulo $\Unip(G)$, then any $y \in G\setminus\Unip(G)$ is a power of $x$ modulo $\Unip(G)$. Thus Proposition \ref{prop:norm:completeness} applies and the result follows.
\end{proof}

We point out that the hypothesis of Proposition \ref{prop:norm:completeness} implies that any element outside $\Unip(G)$ generates a normal subgroup of $G$ by Lemma \ref{lem:norm:completeness}. However the converse of Proposition \ref{prop:norm:completeness} does not hold. Indeed, if $G=$ SmallGroup$(64,28)$ then $\Gamma_{\rm norm}(G)$ is complete and  there exists $x \in G\setminus \Unip(G)$ such that $\langle x \rangle$ is not normal in $G$.

\section{Strong connectivity and diameter}
\noindent In this section we focus on the directed normalizing graph from which its universal vertices have been removed, i.e. $\vec\Delta_{\rm norm}(G)$.  In general, it could be difficult to predict whether this graph is strongly connected or not. For instance, if you consider the $2$-group $G$ as in Example  \ref{esempio}, then 
$\vec\Delta_{\rm norm}(G)$ is strongly disconnected. Thus, we firstly focus on decomposable groups, establishing some conditions that ensure a strong connectivity of $\vec \Delta_{\rm norm}(G)$ and give us a bound on its diameter. Later, we focus on groups with trivial center, since, due to Corollary \ref{cor:norm:A}, 
 for this class of groups $G$ we have $\Univ(G)=\{1\}$ and thus, the vertex set of $\vec\Delta_{\rm norm}(G)$ is $G\setminus \{1\}$. In this case, for soluble groups, we characterize when $\vec\Delta_{\rm norm}(G)$ is strongly connected and establish some bounds on its diameter.

\begin{lemma}\label{lem:3steps}
    Let $G=H\times K$ be a direct product of two non-Dedekind groups. If $(a,b) \in G\smallsetminus \Univ(G)$ and either  $(a,1) \not \in \Univ(G)$ or  $(1,b) \not \in \Univ(G)$ then for every $(c,d) \in G\smallsetminus \Univ(G)$ there is a directed path connecting $(a,b)$ to $(c,d)$ in at most 3 steps.
\end{lemma}

\begin{proof}
If $(1,d) \not \in \Univ(G)$ then  we can consider the path $(a,b) \rightarrow (a,1) \rightarrow (1,d) \rightarrow (c,d)$. If  $(1,d)  \in \Univ(G)$ and $(c,1) \in \Univ(G)$ then by $(i)$ of Lemma \ref{lem:norm:univ-} we have $(c,d) \in \Unip(G)$ and thus $(a,b)\to (c,d)$. Finally, suppose $(1,d) \in \Univ(G)$ and $(c,1) \not \in \Univ(G)$. Since $K$ is not a Dedekind group by Theorem \ref{thm:normcompl} there exists $x \in K \smallsetminus \Univ(K)$. Since $(1,d) \in \Univ(G)$ then $(1,x)$ is normalized by $(c,d)$. Thus, we can consider the path $(a,b) \rightarrow (a,1) \rightarrow (1,x) \rightarrow (c,d)$.
\end{proof}

We are now ready to prove Theorem A.

\begin{proof}[Proof of Theorem A]
From $H$ and $K$ being non-Dedekind groups it follows that $G$ is not a Dedekind group. Let $(a,b), (c,d) \in G\smallsetminus \Univ(G)$. 

First assume $(i)$. Since $(a,b) \not\in \Univ(G)$, we have $(a,1) \not \in \Univ(G)$ or $(1,b) \not \in \Univ(G)$. Thus Lemma \ref{lem:3steps} gives the result. 

Now suppose condition $(ii)$ or $(iii)$ holds. By Lemma \ref{lem:3steps} we can assume that $(a,1), (1,b) \in \Univ(G)$. If both $(c,1)$ and $(1,d)$ are elements of $\Univ(G)$, then by $(i)$ of Lemma \ref{lem:norm:univ-} we have $(a,b) \to (c,d)$.
Without loss of generality assume that $(c,1)\not \in \Univ(G)$. Since $K$ is a non-Dedekind group,  by Theorem \ref{thm:normcompl} there exists $y \in K$ such that  $(1,y)\in G \setminus \Univ(G)$.
If $(ii)$ holds, then $(a,1)\in \Univ(G)$ implies that $a \in Z(H)$. Therefore $(a,b) \rightarrow (c,1)  \rightarrow (c,d)$.
If $(iii)$ holds, then $(1,b)\in \Univ(G)$  implies that $b \in Z(K)$. Therefore $(a,b) \rightarrow (1,y) \rightarrow (c,1) \rightarrow (c,d)$.  

\end{proof}

The hypothesis on $H$ and $K$ being non-Dedekind is necessary in Theorem A. 

\begin{example}
    Let $G=C_3 \times S_3$. Then
\[
\Univ(G)=Z(G)=Z(C_3)\times Z(S_3)=C_3\times\{1\}=\Univ(C_3)\times \Univ(S_3).
\]
Then let $g=(x,(12))$, with $x \in C_3$. Let $h \in G$ such that $g \to h$. Then $g^h=g$ and so $h=(y, (12))$ or $h=(y,1)$, with $y \in C_3$. However, $(y,1)\in \Univ(G)$ and thus it is not a vertex of $\vec\Delta_{\rm norm}(G)$. Therefore $h=(y, (12))$ and this proves that the set $\{(x,(12)) | x \in C_3\}$ is a sink and thus the graph $\vec\Delta_{\rm norm}(G)$ is strongly disconnected.
\end{example}
 It is also possible to find examples where $H$ and $K$ are both non-abelian. Indeed, if you consider $G=Q_8 \times D_{10}$, 

 it follows that $\vec\Delta_{\rm norm}(G)$ is strongly disconnected and $\Univ(G)=Q_8\times\{1\}=\Univ(Q_8)\times \Univ(D_{10})$.

Moreover, the bound in Theorem A is the best possible, as the following example shows.
\begin{example}\label{norm:esempio:2}
    Let $G=S_3\times S_3$. Notice first that $\Univ(S_3)=Z(S_3)=\{1\}$, thus by $(v)$ of Proposition \ref{prop:norm:univ} we have $\Univ(G)=\{1\}=\Univ(S_3)\times \Univ(S_3) $.  Consider $x_1=((12), (23)) \in G$ and $ x_2=((23),(12))\in G$.  Every $x_i$ has order $2$, thus $N_G(\langle x_i \rangle)=C_G(x_i)$. Notice first that $x_1$ and $x_2$ are not adjacent. Moreover, suppose $h=(a,b) \in G \setminus \{x_1\}$ such that $x_1 \rightarrow h$. Then $(12)^a=(12)$ and $(23)^b=(23)$, thus $a\in \{1, (12)\}$ and $b\in \{1, (23)\}$. Therefore, $h=((12), 1)$ or $h=(1, (23))$. In both cases, we do not have $h \to x_2$ and therefore the diameter of $\vec\Delta_{\rm norm}(G)$ is $3$.
\end{example}

Notice that in general $Z(G)=Z(H)\times Z(K) \neq \Univ(G) \neq \Univ(H)\times \Univ(K) $. In fact, let $H$ be the group in Example \ref{esempio} and $K=\langle z \rangle \rtimes Q_8$, $z^3=1$, $z^i=z^2$ and $z^j=z^2$. Then $G=H\times K$ has center of order $8$, while $|\Univ(G)|=12$ and $|\Univ(H)\times \Univ(K)|=24$.

We can say more in the case $K=H$.

\begin{proposition}\label{prop:norm:HxH}
   Let $G=H\times H$, where $H$ is a non-abelian group. Then  $\vec\Delta_{\rm norm}(G)$ is strongly connected of diameter at most $3$.
   Moreover, $\Univ(G)=Z(G)$.
\end{proposition}
\begin{proof}
   If $\Univ(H)=Z(H)$ then from Theorem \ref{thm:normcompl} it follows that $H$ is not a Dedekind group. Moreover, by $(v)$ of Proposition \ref{prop:norm:univ} we have $\Univ(G) = \Univ(H)\times \Univ(H)$. Thus, $G$ satisfies condition $(i)$ of Theorem A, and we are done.
   
   Suppose now  $\Univ(H) \neq Z(H)$.
   Then condition $(ii)$ of Theorem A is satisfied by Lemma \ref{lem:norm:univinH}. 
   If $H$ is not a Dedekind group then the result follows by Theorem A.
   Suppose now $H$ is a Dedekind group. Let $(a,b), (c,d) \in G \setminus \Univ(G)$. Then, using again condition $(ii)$ of Theorem A, we can assume $(a,1) \not \in \Univ(G)$. If $(1,d) \not \in \Univ(G)$ then $(a,b) \leftrightarrow (a,1)\leftrightarrow (1,d) \leftrightarrow (c,d) $ is a path of length at most $3$. If $(1,d)\in \Univ(G)$ then $d \in Z(H)$ and for any $y \in H \setminus Z(H)$ we have $(1,y) \not\in \Univ(G)$ and $(a,b) \leftrightarrow (a,1) \leftrightarrow (1,y) \rightarrow (c,d)$ is the path that connects $(a,b)$ to $(c,d)$ in at most $3$ steps.

Finally,  we prove that $\Univ(G)= Z(G)$. By Proposition \ref{prop:norm:ZUF} one inclusion is obvious. 
Let $(a,b) \in \Univ(G)$. Let $h \in H$. Then $(h,h)^{(a,b)} \in \langle (h,h) \rangle$. Thus, we have $h^a=h^b$. Since $ab^{-1}$ centralizes every $h \in H$, there exists $z \in Z(H)$ such that $b=az$. 
On the other hand, $(a,b) \in \Univ(G)$ implies that there exists an integer $m$ such that $(a,b)^{(h,1)}=(a^m,b^m)=(a^m,(az)^m)$ and $(a,b)^{(h,1)}=(a^h,b)=(a^m,az)$. Thus, $(a^m,(az)^m)=(a^m,az)$. Therefore, $m \equiv 1 \pmod {\mathrm{mcm} (o(a), o(z))}$ and so $m \equiv 1 \pmod{o(a)}$. From this it follows that 
 $a^h=a^m=a$. Thus, $a \in Z(H)$. Since $b=az$, it follows that $b \in Z(H)$ and, therefore, that $(a,b) \in Z(G)$.
\end{proof}

The bound in Proposition \ref{prop:norm:HxH} is the best possible, as Example \ref{norm:esempio:2} shows.

We investigate now the strong connectivity of $\vec\Delta_{\rm norm}(G)$ for a group $G$ with trivial center. Before looking in depth into the main results we first focus on the relationship between the directed normalizing graph of a group and the directed normalizing graph of quotients.

Let $G$ be a group and $N$ a normal subgroup of $G$. Obviously if we take two elements $x, y \in G$ such that $ x^{-1}y \in G\setminus N$ and $x \to y$ in $\vec\Gamma_{\rm norm}(G)$ then $xN \to yN$ in $\vec \Gamma_{\rm norm}(G/N)$. It could be useful to find conditions under which a connection between two elements in $\vec \Gamma_{\rm norm}(G/N)$ gives a connection in $\vec \Gamma_{\rm norm}(G)$.

\begin{proposition}
    Let $G$ be a group and $N$ a normal subgroup of $G$. If $xN \rightarrow yN$  in $\vec \Gamma_{\rm norm}(G/N)$ and $\langle x, y \rangle \cap N =\{1\}$ then $x \rightarrow y$ in $\vec\Gamma_{\rm norm}(G)$.
\end{proposition}
\begin{proof}
   Assume that $yN$ normalizes $\langle xN \rangle$ in $G/N$. This means that $(xN)^{yN}=x^mN$ for some integer $m$. Thus $y^{-1}xy=x^mn$, and so $n=y^{-1}xyx^{-m} \in \langle x, y \rangle \cap N$. Therefore $n=1$ and the result follows.
\end{proof}

We start our investigation on connectivity of $\vec\Delta_{\rm norm}(G)$ by pointing out some connections with other graphs. Recall that the \textit{nilpotent graph} of a group $G$ is the undirected simple graph whose vertices are the elements of $G\setminus Z_{\infty}(G)$, where $Z_{\infty}(G)$ denotes the hypercenter of $G$, and there is an edge between two vertices $x$ and $y$ if and only if $\langle x,y \rangle$ is nilpotent. We will denote it with $\Delta_{\rm nil}(G)$.

\begin{theorem}\label{thm:nil-norm}
Let $G$ be a finite group with trivial center. If $\Delta_{\rm nil}(G)$ is connected then $\vec\Delta_{\rm norm}(G)$ is strongly connected. 
\end{theorem}
\begin{proof}
    Let $x, y \in G \setminus \{1\}$ such that $x$ and $y$ are adjacent in $\Delta_{\rm nil}(G)$. Then $\langle x, y \rangle $ is nilpotent, and thus its center is nontrivial. Consider $z \in Z(\langle x,y \rangle)\setminus \{1\}$. Then $z$ commutes both with $x$ and $y$ and we have $x \leftrightarrow z \leftrightarrow y$. It follows that $\vec\Delta_{\rm norm}(G)$ is strongly connected.
\end{proof}

\begin{corollary}\label{cor:norm}
    Let $G$ be a finite soluble group with trivial center. If $\vec\Delta_{\rm norm}(G)$ is strongly disconnected, then $G$ is a Frobenius group or a $2$-Frobenius group.
\end{corollary}
\begin{proof}
   Assume that  $\vec\Delta_{\rm norm}(G)$ is strongly disconnected. By Theorem \ref{thm:nil-norm} it follows that $\Delta_{\rm nil}(G)$ is disconnected and Theorem A in \cite{DGLM} 
   yields that $G$ is a Frobenius or a $2$-Frobenius group.
\end{proof}

Now we show that  $\vec\Delta_{\rm norm}(G)$ is always strongly disconnected when $G$ is a Frobenius group.

\begin{proposition}\label{prop:norm}
Let $G$ be a Frobenius group. Then $\vec\Delta_{\rm norm}(G)$ is strongly disconnected. 
\end{proposition}
\begin{proof}
    Let $G=K \rtimes H$, with $K$ the Frobenius kernel of $G$ and $H$ a Frobenius complement of $G$. By Lemma \ref{lem:norm1} for any $h \in H\setminus \{1\}$ we have $N_G(\langle h \rangle) \subseteq H$ and thus no element of $G\setminus\{1\}$ outside $H$ normalizes any cyclic subgroup of $H$. Therefore, the induced subgraph of $\vec\Delta_{\rm norm}(G)$ on $H$ is a source, no arrows goes from a vertex of $H$ to a vertex outside $H$, and thus $\vec\Delta_{\rm norm}(G)$ is strongly disconnected. 
\end{proof}

The next two results will be very useful in the sequel.

\begin{lemma}\label{lem:normcentro}
Let $G=KH$ be a finite group such that $K \unlhd G$ is nilpotent and $H$ is cyclic. If there exist a prime $p \in \pi(H)$ and a prime $r \in \pi(K)$ such that $p \mid r-1$, then there exist $k \in Z(K) \setminus \{1\}$ and $h \in H\setminus \{1\}$ such that $k \to h $ in $\vec\Gamma_{\rm norm}(G)$. 
\end{lemma}
\begin{proof}
Let $h \in H$ of order $p$ and $V \leq Z\bigl(O_r(K)\bigr) \leq Z(K) \leq K$ be a minimal normal subgroup of $Z(K)\langle h\rangle $.   Assume by contradiction that $V$ is not cyclic. Then $|V| = r^m$ with $m>1$.  The minimality of $V$ implies that $h$ does not normalize any nontrivial proper subgroup of $V$.

Let $v \in V\setminus \{1\}$. Then $W=\langle v,\,v^h,\,v^{h^2},\dots,v^{h^{p-1}}\rangle \leq V$ and it is normal in $ Z(K)\langle h\rangle $. Thus, $V=W$. If $V$ has order $r^p$ then $w = v\cdot v^h\cdot v^{h^2}\cdots v^{h^{p-1}}$ is nontrivial and it is centralized by $h$, a contradiction. Therefore there exists a natural number $2 \leq \ell \leq p-1$ such that $|V| = r^\ell$.

Let $\mathcal{C}$ be the set of all nontrivial cyclic subgroups of $V$.  As $p\mid r-1$, we have $r\equiv 1\pmod p$. Obviously there are $r^\ell -1$ nontrivial elements in $V$ and for every nontrivial element $x$ of $V$ the subgroup $\langle x \rangle$ is a cyclic subgroup with $r-1$ generators. It follows that 
\[
|\mathcal{C}|
=\frac{r^\ell-1}{r-1}
=r^{\ell-1} + r^{\ell-2} + \cdots + r + 1
\equiv \ell \pmod p.
\]
Moreover, consider the action of $\langle h \rangle $ on $\mathcal{C}$ by conjugation. Thus every orbit of the action has size $p$.  So 
\[
|\mathcal{C}|\equiv 0 \pmod p.
\]
 It follows that $\ell \equiv 0 \pmod p$, a contradiction. 

 Thus $V$ is cyclic. Then, $h$ normalizes $V$ and we have a connection between an element of $Z(K)\setminus\{1\}$ and $h$. 

\end{proof}

\begin{corollary}\label{cor:normcentro}
    Let $G$ be a Frobenius group with kernel $K$ and cyclic complement $H$. If there exist a prime $p \in \pi(H)$ and a prime $r \in \pi(K)$ such that $p \mid r-1$,  then there exists a path in $\vec\Delta_{\rm norm}(G)$ connecting a vertex in $Z(K)$ to any vertex in $H^g$ in at most $2$ steps for all $g \in G$.
\end{corollary}

\begin{proof}
    Let $G$ be a Frobenius group. By Lemma \ref{lem:normcentro} there exist elements $k \in Z(K)\setminus \{1\}$ and $h \in H\setminus \{1\}$ such that $k^h=k^n $, for some integer $n$. Let $g \in G$. Obviously $h^g \in H^g\setminus\{1\}$. Then $k^g \in Z(K)\setminus\{1\}$ and 
\[
(k^g)^{h^g}=(k^h)^g=(k^n)^g=(k^g)^n,
\]
thus $h^g$ normalizes $\langle k^g \rangle$. Since $H^g$ is cyclic, $k^g$ is connected to every vertex in $H^g$ in at most $2$ steps and we are done.
\end{proof}

We now start our investigation on $2$-Frobenius groups, with the goal of characterizing when $\vec \Delta_{\rm norm}(G)$ for such a group is strongly disconnected.

\begin{proposition}\label{prop:normcomp}
    Let $G$ be a $2$-Frobenius group as in Definition \ref{2-Frob}.  Then 
    \begin{enumerate}
        \item[$(i)$] for any $g \in G\setminus HK$ there exists $k \in K$ such that $g \leftrightarrow g_{[p]} \leftrightarrow k$;
        \item[$(ii)$] $G\setminus X$ lies in a strongly connected component of $\vec \Delta_{\rm norm}(G)$;
        \item[$(iii)$] for any $x \in X \setminus \{ 1\}$ there exists an element $y_{[p]} \in G \setminus HK$ of prime power order such that $x \to y_{[p]}$.
    \end{enumerate}
\end{proposition} 

\begin{proof}
To prove $(i)$, let $g \in G\setminus HK$. Then $g \leftrightarrow g_{[p]}$. Of course $g_{[p]}$ is outside $HK$, otherwise $g_{[p]}K$ would lie in the kernel but also in the complement of the Frobenius group $G/K$. By Lemma \ref{lem:norm2} $C_K(g_{[p]}) \neq \{1\}$, therefore there exists an element $k \in K\setminus\{1\}$ such that $g_{[p]} \leftrightarrow k$.

    To show $(ii)$, let $g \in G\setminus X$ such that $g \not \in K$. By $(i)$, there exists $k \in K \setminus \{1\}$ such that $g \leftrightarrow g_{[p]} \leftrightarrow k$. Since $K$ has nontrivial center, there exists $z \in Z(K)$ such that  $k_1 \leftrightarrow z \leftrightarrow k$ for all $k_1 \in K$ . Thus there exists a path that connects every element of $G\setminus X$ with every other element of $G\setminus X$, and item $(ii)$ is proved.
    
    As regards item $(iii)$, let $g \in G \setminus HK$. Then $H^gK=H K$, so there exists $k \in K$ such that $H^g=H^k$ by Lemma \ref{lem:norm2}. As a consequence $y=g k^{-1} \in N_G(H) \setminus HK$. In particular, there exists $y_{[p]}$ of $y$ in $G\setminus HK$, which normalizes $H$, too.  As $H$ is cyclic, all its subgroups are characteristic in $H$, so for any $h_1 \in H$ we have $h_1 \to y_{[p]}$. By definition, any element of $X$ belongs to some conjugate of $H$. Thus, taking a suitable conjugate of $y_{[p]}$, the result follows. 
\end{proof}

 For a group $G$, we denote with $\Delta_{\rm norm}(G)$ the undirected graph induced by $\vec \Delta_{\rm norm}(G)$.

\begin{proposition}\label{prop:normcaratt}
    Let $G$ be a $2$-Frobenius group as in Definition \ref{2-Frob}.     Then $\vec\Delta_{\rm norm}(G)$ is strongly disconnected if and only if $p  \nmid r-1$ for all $p \in \pi(H)$ and for all $r \in \pi(K)$. 
\end{proposition}
\begin{proof}
         
     First assume that $p  \nmid r-1$ for all $p \in \pi(H)$ and for all $r \in \pi(K)$. Then $\Delta_{\rm norm}(HK)$ is strongly disconnected by Proposition~10 of \cite{FarrellParker2025}. It immediately follows that no arrows goes from any vertex in $K$ to a vertex in any conjugate $H^g$. Finally, by contradiction assume that for some $g \in G\setminus HK$ there exists $x \in X \setminus\{1\}$ with $g \to x$. Then $\langle gK \rangle$ is normalized by $xK$ in $G/K$. Since $gK$ lies in a complement of $G/K$, Lemma \ref{lem:norm1} forces $xK=K$, hence $x \in K$, which is a contradiction. Therefore $X$ is a source set, and $\vec\Delta_{\rm norm}(G)$ is  strongly disconnected.

Conversely, assume that $\vec\Delta_{\rm norm}(G)$ is strongly disconnected. Then, it suffices to show that $\Delta_{\rm norm}(HK)$ is strongly disconnected, as the result will follow from Proposition 10 of \cite{FarrellParker2025}. Thus, by way of contradiction assume that $\Delta_{\rm norm}(HK)$ is strongly connected. Then for any $x_1,x_2 \in X \setminus\{1\}$ there is a path connecting $x_1$ to $x_2$. Indeed, Lemma \ref{lem:norm1} implies the existence of $k_2 \in K$ such that $k_2 \to x_2$. Moreover, by items $(i)$ and $(iii)$ of Proposition \ref{prop:normcomp} there exist $g_{[p]} \in G \setminus HK$ of prime order and $k_1 \in K \setminus \{1\}$ such that $x_1 \to g_{[p]} \leftrightarrow k_1$. As the center of $K$ is nontrivial, there exists $z \in Z(K) \setminus \{1\}$ such that $x_1 \to g_{[p]} \to k_1 \to z \to k_2 \to x_2$. This, together with item $(ii)$ of Proposition \ref{prop:normcomp}, implies that $\vec\Delta_{\rm norm}(G)$ is strongly connected, which is a contradiction.
\end{proof}

We can now characterize when $\vec\Delta_{\rm norm}(G)$ is strongly disconnected for a soluble group $G$ with trivial center, proving Theorem B.

\begin{proof}[Proof of Theorem B]
    If $G$ is a Frobenius group or a 2-Frobenius group as in $(ii)$, then $\vec\Delta_{\rm norm}(G)$ is strongly disconnected  by Proposition \ref{prop:norm} and Proposition \ref{prop:normcaratt}, respectively.

    Conversely, if $\vec\Delta_{\rm norm}$ is strongly disconnected, then $G$ is a Frobenius or a $2$-Frobenius group by Corollary \ref{cor:norm}. Suppose $G$ is a $2$-Frobenius group, with $K \lhd KH \lhd G $, $KH$ a Frobenius group and $G/K$ a Frobenius group with kernel $KH/K$. Then, by Proposition \ref{prop:normcaratt}, $\vec\Delta_{\rm norm}$ strongly disconnected implies that for all $p \in \pi(H)$ and for all $r \in \pi(K)$ we have $p \nmid r-1$. This concludes the proof.
\end{proof}

Now we focus on the diameter of  $\vec\Delta_{\rm norm}(G)$ when it is strongly connected.

\begin{proposition}\label{prop:2frobconnected}
    Let $G$ be a $2$-Frobenius group as in Definition \ref{2-Frob}. If $\vec\Delta_{\rm norm}(G)$ is strongly connected then ${\rm diam}(\vec\Delta_{\rm norm}(G)) \leq 6$.
\end{proposition}
\begin{proof}
From Proposition \ref{prop:normcaratt} there exist $p \in \pi(H)$ and $r \in \pi(K)$ such that $p  \mid r-1$.

Let $x,y \in G \setminus \{1\}$. If $x,y \in K$, then $x \leftrightarrow z  \leftrightarrow y$ for any nontrivial $z \in Z(K)$.

If $x \in K$ and $y \in X$, then by Corollary \ref{cor:normcentro} for any vertex $z \in Z(K)$ there exists a path connecting $z$ to $y$ in at most two steps. Thus $x\to z \to w \to y$ for a suitable $w$ and we are done. Moreover, by Proposition \ref{prop:normcomp} $(iii)$ and Lemma \ref{lem:norm2} there exist vertices $g_{[p]} \in G \setminus HK$, $k_2 \in K$ and $z_1 \in Z(K)$ such that $y \to g_{[p]} \to k_2 \to z_1 \to x$.

If $x \in K $ and $y \in G \setminus HK$, by Lemma \ref{lem:norm2} there exist vertices $y_{[p]} \in G \setminus HK$, $k \in K$ and $z \in Z(K)$ such that $y \leftrightarrow y_{[p]} \leftrightarrow k \leftrightarrow z \leftrightarrow x$.

If $x,y \in X$, then applying Lemma \ref{lem:normcentro} and Proposition \ref{prop:normcomp} $(iii)$ there exist vertices $z \in Z(K)$, $k \in K$, $g_{[p]} \in G \setminus HK$, $y_1 \in X$  such that $x \to g_{[p]} \to k \to z \to y_1 \to y$.

Let $x \in X$ and $y \in G \setminus HK$. On the one hand, Proposition \ref{prop:normcomp} and Lemma \ref{lem:norm2} implies the existence of vertices $g_{[p]}, y_{[q]} \in G \setminus HK$, $k,k_1 \in K$, $z \in Z(K)$ such that $x \to g_{[p]} \to k \to z \leftrightarrow k_1 \leftrightarrow y_{[q]} \leftrightarrow y$. On the other hand, Lemma \ref{lem:normcentro} ensures the existence of a vertex $k_2 \in Z(K)$ and $x_1 \in X$ such that $k_2 \to x$. Thus $y \to y_{[q]} \to k_1 \to k_2 \to x_1 \to x$.

 If $x,y \in G\setminus HK$ then by Proposition \ref{prop:normcomp} there exist $k, k_1 \in K$ and $z \in Z(K)$ such that $x \leftrightarrow x_{[p]} \leftrightarrow k \leftrightarrow z \leftrightarrow k_1 \leftrightarrow y_{[q]} \leftrightarrow y$.
\end{proof}

We recall that the {\it commuting graph} $\Delta_{\rm comm}(G)$ of a group $G$ is the simple and undirected graph whose vertices are all elements of $G\setminus Z(G)$ and two vertices are adjacent when they commute. Observe that if $x$ and $y$ are adjacent in $\Delta_{\rm comm}(G)$ then $x \leftrightarrow y$ in $\vec\Gamma_{\rm norm}(G)$. As usual, we refer to a group whose Sylow subgroups are abelian as an $A$-group.

\begin{proposition}
    Let $G$ be a finite soluble $A$-group with trivial center and $\vec \Delta_{\rm norm} (G)$ strongly connected. Then ${\rm diam}(\vec \Delta_{\rm norm} (G)) \leq 6$.
\end{proposition}
\begin{proof}
Since $\vec \Delta_{\rm norm} (G)$ is strongly connected, $G$ is not Frobenius. If $G$ is $2$-Frobenius, then the result follows by Proposition \ref{prop:2frobconnected}. If $G$ is not $2$-Frobenius, then $\Gamma_{\rm comm}(G)$ is connected of diameter at most $6$ by   \cite[Theorem 1.1]{CL-agrp}. Therefore $\vec \Delta_{\rm norm} (G)$ is strongly connected too, and the result follows.
\end{proof}

We point out that Theorem 1.1 of \cite{parker} provides a general bound on the diameter of $\vec\Delta_{\rm norm}(G)$ for a soluble group $G$ with trivial center.

\begin{theorem}
Let $G$ be a finite soluble group with trivial center. If $\vec\Delta_{\rm norm}(G)$ is strongly connected then the diameter of $\vec\Delta_{\rm norm}(G)$ is at most $8$.
\end{theorem}
\begin{proof}
Since $\vec\Delta_{\rm norm}(G)$ is strongly connected, $G$ is not Frobenius by Proposition~\ref{prop:norm}. If $G$ is 2-Frobenius, then $diam(\vec\Delta_{\rm norm}(G)) \leq 6$ by Proposition \ref{prop:2frobconnected}. Now assume that $G$ is not 2-Frobenius. Then $\Gamma_{\rm comm}(G)$ is strongly connected and $\mathrm{diam}(\Gamma_{\rm comm}(G)) \leq 8$ by Theorem 1.1 \cite{parker}. Then $\mathrm{diam}(\vec\Delta_{\rm norm}(G))\leq \mathrm{diam}(\Gamma_{\rm comm}(G))$ and the result follows.
\end{proof}

It is still unknown whether this bound is sharp. However, we point out that in \cite{FarrellParker2025} Farrell and Parker provided an example of a soluble group $G$ for which $\Delta_{\rm norm}(G)$ is connected of diameter $6$. Since $G$ is neither a Frobenius group nor a $2$-Frobenius group,  then $\vec \Delta_{\rm norm}(G)$ is strongly connected and therefore its diameter is at least $6$.

We investigate now the case in which $\vec \Delta_{\rm norm} (G)$ is strongly disconnected, finding the number of strongly connected components and a bound on their diameters.

\begin{thm}
Let $G$ be a finite soluble group with trivial center and suppose that $\vec \Delta_{\rm norm} (G)$ is strongly disconnected.  Then the number of strongly connected components is $|\mathrm{Fit}(G)|+1$; moreover, one strongly connected component has diameter at most $6$ and all other strongly connected components have diameters at most $2$.
\end{thm}
\begin{proof}
    By Theorem B, $G$ is a Frobenius or a $2$-Frobenius group. Suppose first that $G=K H$ is a Frobenius group. Then no arrows goes from a vertex in $K$ to a vertex in any $H^g$ and by Lemma \ref{lem:norm1} no arrows goes form a vertex in any $H^g$ to a vertex in $K$. Moreover, $K$ is nilpotent and therefore its center is nontrivial, thus the subgraph induced by $K$ is a strongly connected component of $G$. By Lemma \ref{lem:norm1} and the fact that the center of a Frobenius complement is nontrivial it follows that any of the conjugates of $H$ gives rise to a strongly connected component. Finally, since $K=\mathrm{Fit}(G)$, there are $|\mathrm{Fit}(G)|+1$ strongly connected components of diameter at most $2$. Now suppose that $G$ is a $2$-Frobenius group as in Definition \ref{2-Frob}. 
    By Proposition~\ref{prop:normcomp} the subgraph induced by $G\setminus X$ lies in a strongly connected component and there exists a path from any vertex in $X$ to some $g \in G \setminus X$. Therefore, there is no arrow from any element of $G \setminus X$ to $X\setminus\{1\}$, as $\vec \Delta_{\rm norm} (G)$ is strongly disconnected.
    Thus the subgraph induced by $G\setminus X$ coincides with a strongly connected component. One can easily see that any conjugate of $H$ is a strongly connected component of diameter 1. 
    As $K=\mathrm{Fit}(G)$, there are exactly  $|\mathrm{Fit}(G)|+1$ strongly connected components. Using Proposition~\ref{prop:normcomp} $(i)$ there is a path from any element of $G \setminus HK$ to any vertex in $Z(K)$ of length at most 3. Thus, we can connect any two elements of $G\setminus HK$ in at most 6 steps.
\end{proof}

We point out that Theorem \ref{thm:normcompl} also characterizes when $\vec \Gamma_{\rm norm} (G)$ is strongly connected of diameter $1$. It is not difficult to describe also when $\vec \Delta_{\rm norm} (G)$ is strongly disconnected with all its strongly connected components of diameter $1$, in the case in which $G$ is a soluble group with trivial center.

\begin{proposition}
    Let $G$ be a finite soluble group with trivial center. If $\vec \Delta_{\rm norm} (G)$ is strongly disconnected with strongly connected components of diameter $1$, then  $G$ is a Frobenius group with Dedekind Frobenius kernel and Dedekind Frobenius complement.
\end{proposition}
\begin{proof}

By Theorem B it follows that $G$ is either a Frobenius or a $2$-Frobenius group. Suppose, by contradiction, that $G$ is a $2$-Frobenius group as in Definition \ref{2-Frob}. Then by Proposition \ref{prop:normcomp}, $G\setminus X$ lies in a strongly connected component of $\vec \Delta_{\rm norm} (G)$. Therefore, $a \leftrightarrow b$ for any  $a,b \in G\setminus X$, which implies that $aK$ normalises $\langle bK \rangle$. However this is a contradiction  by Lemma \ref{lem:norm1}.  Therefore $G$ is a Frobenius group, whose complement and kernel are Dedekind by Theorem \ref{thm:normcompl}.
\end{proof}

The bound in Proposition \ref{prop:2frobconnected} can be improved under certain conditions.

\begin{proposition}
    Let $G$ be $2$-Frobenius group 
    as in Definition \ref{2-Frob} and $\vec \Delta_{\rm norm}(G)$ be strongly connected. If $\pi(K)=\pi(L)$ or for any prime $p\in \pi(L)\setminus \pi(K)$ there exists a prime $r \in \pi(K)$ such that $p \mid r-1$ then ${\rm diam}( \vec \Delta_{\rm norm}(G)) \leq 5$.
\end{proposition}
\begin{proof}
Following the proof of Proposition \ref{prop:2frobconnected}, we only need to show that $x,y$ are at directed distance at most $5$ when $x\in X$ and $y \in G\setminus HK$ and when $x,y \in G\setminus HK$. In both cases it suffices to prove that for any $g \in G\setminus HK$ there exists $z \in Z(K)$ such that $z$ reaches $g$ in at most $2$ steps.

Let $g \in G\setminus HK$. If there exists a prime $p$ such that $p \mid o(g)$ and $p \mid |K|$, then for any Sylow $p$-subgroup  $P$ of $G$ containing $g_{[p]}$
we have $z \leftrightarrow g_{[p]} \leftrightarrow g$, where $z$ is a nontrivial element of $Z(P) \cap Z(O_p(G)) \leq Z(K)$.

Assume now that $(o(g), |K|)=1$. By hypothesis, if $p$ is a prime dividing $o(g)$ then there exists a prime $r$ such that $p \mid r-1$. Consider $K\langle g \rangle $. By Lemma \ref{lem:normcentro}, there exists an element $z \in Z(K)\setminus\{1\}$ such that $z$ reaches $g$ in at most $2$ steps. This concludes the proof.

\end{proof}

In the following, we show that the upper bound on the diameter of $\vec \Delta_{\rm norm} (G)$ can be improved for some classes of groups.
Recall that a group $G$ is a cyclic-by-abelian if $G$ has a cyclic normal subgroup $N$ such that the quotient $G/N$ is abelian.

\begin{proposition}\label{prop:norm:cyclic-by-abelian}
   Let $G$ be a finite cyclic-by-abelian group with trivial center and $\vec \Delta_{\rm norm} (G)$ strongly connected. Then ${\rm diam}(\vec \Delta_{\rm norm} (G)) \leq 4$.
\end{proposition}
\begin{proof}
   Let $N=\langle c \rangle$ be the cyclic normal subgroup of $G$ such that $G/N$ is abelian and let $g \in G$. Since $N$ is cyclic, every subgroup of $N$ is normal in $G$ and so $N \subseteq \Unim(G)$. Thus, it suffices to show that $g$ can reach an element of $N\setminus\{1\}$ in at most $3$ steps.
 If $g\in N$ we are done, thus suppose that $g \not \in N$.
 
   If $(|N|,o(g))\neq 1$, take a prime $p$ dividing both $|N|$ and $o(g)$. Thus, there exists a Sylow $p$-subgroup $P$  of $G$ containing $g_p$. Now, for a $z \in Z(P)$ and a suitable positive integer $m$ we have $g \to g_p \to z \to c^m$, with $1 \neq c^m \in N \cap P$.
   
 Now assume $(|N|,o(g)) = 1$ and let $p$ be a prime dividing $o(g)$. Then, there exists a Sylow $p$-subgroup $P$ of $G$ containing $g_p$. Since $PN/N$ is normal in $G/N$,  $PN$ is normal in $G$. Thus, by Frattini's argument $G=NN_G(P)$. If $N \cap N_G(P)=\{1\}$ then $N_G(P)$ is abelian. Moreover, since $\vec \Delta_{\rm norm} (G)$ is strongly connected, by Proposition \ref{prop:norm} it follows that $G$ is not a Frobenius group, and so there exists $h \in N$ such that $C_{N_G(P)}(h)\neq \{1\}$. Let $t \in C_{N_G(P)}(h)\setminus\{1\}$. Thus, we have $g \to g_p \to t \to h$. Now assume $N \cap N_G(P)\neq \{1\}$ and $1 \neq u \in N \cap N_G(P)$. Since $N \cap N_G(P)$ is normal in $N_G (P)$, we have $N \cap N_G(P) \subseteq {\rm Fit} (N_G(P))$.  Moreover, $P \subseteq {\rm Fit} (N_G(P))$, so $\langle g_p, u \rangle$ is nilpotent.  Thus we have $g \to g_p \to z \to u$, where $z$ is any nontrivial element in $Z(\langle g_p, u \rangle)$.
\end{proof}

\begin{corollary}
    Let $G$ be a finite soluble group with trivial center and $\vec \Delta_{\rm norm} (G)$ strongly connected. If ${\rm Fit}(G)$ is cyclic then ${\rm diam}(\vec \Delta_{\rm norm} (G)) \leq 4$.
\end{corollary}
\begin{proof}
     By Proposition \ref{prop:norm:cyclic-by-abelian}, it suffices to prove that $G/F$ is abelian. Notice that since $G$ is a finite soluble group we have $C_G(F)=Z(F)=F$. Thus $N_G(F)/C_G(F)=G/F$ is isomorphic to a subgroup of ${\rm Aut} (F)$ which is abelian, since $F$ is cyclic; therefore, $G/F$ is abelian. This proves the result.    
\end{proof}

The same bound holds when the Fitting subgroup of $G$ has prime index.
 
\begin{proposition}\label{prop:norm:primeindex}
    Let $G$ be a finite soluble group with trivial center and $\vec \Delta_{\rm norm} (G)$ strongly connected. If $|G:{\rm Fit}(G)|$ is a prime number then ${\rm diam}(\vec \Delta_{\rm norm} (G)) \leq 4$.
\end{proposition}
\begin{proof}
   Let $F=\mathrm{Fit}(G)$ and $|G:F|=p$. We will prove that there exist paths connecting every element of $G\setminus\{1\}$ to any element of $Z(F)\setminus\{1\}$ and vice versa in at most $2$ steps.
   
   Let $x \in G\setminus \{1\}$. If $x \in F$ we are done. Assume $x \in G\setminus F$. If $x^p \neq 1$ then $x^p \in F\setminus\{1\}$ and so we have the path $x \leftrightarrow x^p \leftrightarrow z $, for all $z \in Z(F)$. 
   
   If $x^p=1$ then $x$ lies in Sylow $p$-subgroup of $G$, say $P$. If $p$ divides $|F|$ then we have $Z(O_p(G)) \leq P $. Thus there exists $z_1 \in Z(O_p(G)) \leq Z(F)$ such that  $x \leftrightarrow z_1 $. If $p$ does not divide $|F|$ then $G=F \rtimes \langle x \rangle $. Since $\vec \Delta_{\rm norm} (G) $ is strongly connected, due to Proposition \ref{prop:norm} it follows that $G$  is not a Frobenius group and thus there exists a nontrivial element $y \in \langle x \rangle$ that centralizes a nontrivial element of $F$, say $f$. Hence $x$ itself commutes with $f$ and so we have the path $x \leftrightarrow f \leftrightarrow z$ for all $z \in Z(F)\setminus\{1\}$. This concludes the proof.
\end{proof}

The bound in Proposition \ref{prop:norm:primeindex} is sharp. Indeed the group $G=$ SmallGroup(384,591) has  Fitting subgroup of order $128$ and index $3$, and $\vec \Delta_{\rm norm} (G)$ is strongly connected of diameter 4.

\section*{Acknowledgements}
\noindent The authors are members of the National Group for Algebraic and Geometric Structures,
and their Applications (GNSAGA – INdAM). This research has beenfunded by the European Union - Next Generation EU, Missione 4 Componente 1 CUP B53D23009410006, PRIN 2022 - 2022PSTWLB - Group Theory and Applications.

\section*{Availability of data materials} This manuscript has no associated data.

\end{document}